\pdfoutput=1
\RequirePackage{ifpdf}
\ifpdf 
\documentclass[pdftex]{sigma}
\else
\documentclass{sigma}
\fi

\usepackage[all]{xy}
\CompileMatrices

\newdir{ >}{!/10pt/@{ }*@{>}}
\newdir{< }{!/10pt/@{ }*@{<}}

\usepackage{paralist}

\numberwithin{equation}{section}

\newtheorem{theo}{Theorem}[section]
\newtheorem{coro}[theo]{Corollary}
\newtheorem{Lemma}[theo]{Lemma}
\newtheorem{propo}[theo]{Proposition}
 { \theoremstyle{definition}
\newtheorem{Remark}[theo]{Remark} }

\newcommand{\RR}{\mathbb{R}}

\newcommand{\ZZ}{\mathbb{Z}}

\usepackage{bm}
\newcommand{\simbolovettore}[1]{{\boldsymbol{#1}}}

\newcommand{\vn}{\simbolovettore{n}}

\newcommand{\vq}{\simbolovettore{q}}

\newcommand{\vu}{\simbolovettore{u}}
\newcommand{\vv}{\simbolovettore{v}}
\newcommand{\vw}{\simbolovettore{w}}
\newcommand{\vx}{\simbolovettore{x}}

\newcommand{\vz}{\simbolovettore{z}}
\newcommand{\vQ}{\simbolovettore{Q}}
\newcommand{\vN}{\simbolovettore{N}}

\newcommand{\zero}{\boldsymbol{0}}

\newcommand{\im}{\operatorname{Im}}

\DeclareMathOperator{\Ima}{Im}

\newcommand{\mgrad}{\nabla_M}
\newcommand{\hatmgrad}{\hat\nabla_M}
\newcommand{\conf}[2]{\mathbb{F}_{#1}(#2)}

\newcommand{\eucnorm}[1]{\lvert{#1}\rvert}
\newcommand{\mnorm}[1]{{\lVert{#1}\rVert}_M}
\newcommand{\mscalar}[2]{{\langle{#1},{#2}\rangle}_{M}}

\begin{document}

\allowdisplaybreaks

\newcommand{\arXivNumber}{1608.00480}

\renewcommand{\PaperNumber}{021}

\FirstPageHeading

\ShortArticleName{Central Conf\/igurations and Mutual Dif\/ferences}

\ArticleName{Central Conf\/igurations and Mutual Dif\/ferences}

\Author{D.L.~FERRARIO}

\AuthorNameForHeading{D.L.~Ferrario}

\Address{Department of Mathematics and Applications, University of Milano-Bicocca,\\ Via R.~Cozzi, 55 20125 Milano, Italy}
\Email{\href{mailto:davide.ferrario@unimib.it}{davide.ferrario@unimib.it}}
\URLaddress{\url{http://www.matapp.unimib.it/~ferrario/}}

\ArticleDates{Received December 06, 2016, in f\/inal form March 27, 2017; Published online March 31, 2017}

\Abstract{Central conf\/igurations are solutions of the equations $\lambda m_j\vq_j = \frac{\partial U}{\partial \vq_j}$, where~$U$ denotes the potential function and each $\vq_j$ is a point in the $d$-dimensional Euclidean space $E\cong \RR^d$, for $j=1,\ldots, n$. We show that the vector of the mutual dif\/ferences \smash{$\vq_{ij} = \vq_i - \vq_j$} satisf\/ies the equation $-\frac{\lambda}{\alpha} \vq = P_m(\Psi(\vq))$, where~$P_m$ is the orthogonal projection over the spaces of $1$-cocycles and $\Psi(\vq) = \frac{\vq}{\eucnorm{\vq}^{\alpha+2}}$. It is shown that dif\/ferences~$\vq_{ij}$ of central conf\/igurations are critical points of an analogue of~$U$, def\/ined on the space of $1$-cochains in the Euclidean space~$E$, and restricted to the subspace of $1$-cocycles. Some generalizations of well known facts follow almost immediately from this approach.}

\Keywords{central conf\/igurations; relative equilibria; $n$-body problem}

\Classification{37C25; 70F10}

\section{Introduction}

Central conf\/igurations play an important role in the (Newtonian) $n$-body problem: to name two, they arise as conf\/igurations yielding homographic solutions, and as rest points in the f\/low on the McGehee collision manifold. Following the spirit of Albouy and Chenciner~\cite{MR1489897}, in this article we study the problem of central conf\/igurations from the point of view of mutual distances; but instead of lengths we consider the space of \emph{differences} of positions, which turns out to be a~suitable group of cochains~$C^1$ of degree $1$ with coef\/f\/icients in the Euclidean space~$E$. Hence, we show that central conf\/igurations are critical points of a function def\/ined on $C^1$ and restricted to the subspace of $1$-cocycles, and show some consequences. The technique of embedding the central conf\/igurations problem into a suitable space of cocycles was actually already used by Moeckel in \cite{MR805839}, in an implicit way, and again by Moeckel and Montgomery in~\cite{MR3069058}. In this article we study this approach introducing cocycles and cohomology, and show that many calculations can be signif\/icantly simplif\/ied in this way. For further details and recent remarkable advances we refer to~\cite{MR2925390, MR2207019}.

More precisely, assume $n\geq 2$, $d\geq1$. Let $E=\RR^d$ denote the $d$-dimensional Euclidean space. An element of $E^n$ will be denoted by $\vq=(\vq_1,\vq_2,\ldots, \vq_n)$ where $\forall\, j$, $\vq_j \in E$. Let $\conf{n}{E}$ denote as in~\cite{FH} the conf\/iguration space of~$n$ particles in~$E$:
\begin{gather*}
\conf{n}{E} = \big\{\vq \in E^n \colon \vq_i \neq \vq_j \big\}.
\end{gather*}
If $\Delta$ is the collision set
\begin{gather*}
\Delta = \bigcup_{i<j} \big\{\vq \in E^n \colon \vq_i = \vq_j \big\},
\end{gather*}
then $\conf{n}{E} = E^n \setminus \Delta$.

For $j=1,\ldots, n$, let $m_j >0$ be positive masses. Assume that the masses are normalized, i.e., that
\begin{gather}\label{eq:masses1}
\sum_{j=1}^n m_j = 1.
\end{gather}
Let $\mscalar{*}{*}$ denote the \emph{mass-metric} on (the tangent vectors of) $E^n$, def\/ined as
\begin{gather*}
\mscalar{\vv}{\vw} = \sum_{j=1}^n m_j \vv_j \cdot \vw_j,
\end{gather*}
where $\vv_j \cdot \vw_j$ is the Euclidean scalar product in (the tangent space of)~$E$. Let $\eucnorm{\vv_j}$ denote the Euclidean norm of a~vector $\vv_j$ in $E$. The norm corresponding to the mass-metric is $\mnorm{\vv} = \sqrt{\mscalar{\vv}{\vv}}$.

Let $\alpha>0$ be a f\/ixed homogeneity parameter, and $U \colon \conf{n}{E} \to \RR$ the potential function def\/ined as
\begin{gather*}
U(\vq) = \sum_{1\leq i<j \leq n} \frac{m_im_j}{\eucnorm{\vq_i - \vq_j}^\alpha}.
\end{gather*}

A \emph{central configuration} is a conf\/iguration $\vq\in \conf{n}{E}$ such that there exists $\lambda \in \RR$ such that ($\forall\, j=1,\ldots, n$)
\begin{gather}
\label{eq:cc1}
\lambda m_j \vq_j = \frac{\partial U}{\partial \vq_j} =
- \alpha \sum_{k\neq j} m_jm_k\frac{\vq_j - \vq_k}{\eucnorm{\vq_j-\vq_k}^{\alpha+2}}.
\end{gather}
If $\vq$ is a central conf\/iguration, then
\begin{gather*}
\begin{aligned}
\lambda \sum_{j} m_j \eucnorm{\vq_j}^2 &
= -\alpha \sum_{j=1}^n \sum_{k\neq j,k=1}^{n} m_j m_k \frac{(\vq_j - \vq_k)\cdot \vq_j}{|\vq_j - \vq_k|^{\alpha+2}}\\
& = - \alpha\sum_{j<k} m_jm_k \frac{ (\vq_j - \vq_k)\cdot \vq_j + (\vq_k - \vq_j)\cdot \vq_k} {\eucnorm{\vq_j-\vq_k}^{\alpha+2}} =
 - \alpha
\sum_{j<k} \frac{m_jm_k} {\eucnorm{\vq_j-\vq_k}^{\alpha}} \\
\implies \lambda \mnorm{\vq}^2 & = -\alpha U(\vq),
\end{aligned}
\end{gather*}
and hence $\lambda = -\alpha \frac{ U(\vq)}{\mnorm{\vq}^2} <0$. By summing equation~\eqref{eq:cc1} in $j$
\begin{gather*}
\lambda \sum_j m_j \vq_j = - \alpha\sum_{j<k} m_jm_k \frac{ (\vq_j - \vq_k) + (\vq_k - \vq_j)} {\eucnorm{\vq_j-\vq_k}^{\alpha+2}} = \zero
\end{gather*}
and hence
\begin{gather*}
\sum_j m_j \vq_j = \zero.
\end{gather*}
For an analysis of central conf\/igurations for general potential functions $U(\vq)$, see \cite{MR2372989,Fer2015}. Also, central conf\/igurations can be equivalently seen as:
\begin{compactenum}
\item[\textbf{(CC1)}]\label{item:cc1} Solutions of \eqref{eq:cc1} \cite{MR1082871}.
\item[\textbf{(CC2)}]\label{item:cc2} Critical points of the restriction of the potential function $U$ to the inertia ellipsoid $S = \{ \vq \in \conf{n}{E} \colon \mnorm{\vq}^2=1 \}$ \cite{MR3469182}.
\item[\textbf{(CC3)}]\label{item:cc3} Fixed points of the map $F\colon S \to S$ def\/ined as $F(\vq) = - \frac{ \nabla_M U(\vq)}{\mnorm{\nabla_M U(\vq)}}$,
where $\nabla_M$ denotes the gradient with respect to the mass-metric on ${\conf{n}{E}}$ \cite{MR2372989,Fer2015}.
\item[\textbf{(CC4)}]\label{item:cc4} Critical points on $\conf{n}{E}$ of the map $\mnorm{\vq}^2+ U(\vq)$~\cite{ItuMad2015}.
\item[\textbf{(CC5)}]\label{item:cc5} Critical points on $\conf{n}{E}$ of the map $\mnorm{\vq}^{2\alpha} U(\vq)^2$ (or $\mnorm{\vq}^\alpha U(\vq)$)~\cite{MR2052612}.
\end{compactenum}

In all these formulations, central conf\/igurations appear as $O(d)$-orbits in $\conf{n}{E}$, where the action of the orthogonal group $O(d)$ on~$\conf{n}{E}$ is diagonal $g\cdot \vq = (g\vq_1,\ldots, g\vq_n)$.

Def\/ine the space $X$ as
\begin{gather*}
X = \left\{ \vq \in \conf{n}{E} \colon \sum_{j=1}^n m_j \vq_j = \zero \right\}.
\end{gather*}

\section{Central conf\/igurations and mutual dif\/ferences}

Let $\vn$ be the set $\vn=\{1,2,\ldots, n\}$ and $C^0$ the vector space of all maps from $\vn$ to $E$: $C^0 = \{ \vq \colon \vn\to E \}$. Let $\conf{n}{E} \subset C^0$ denote the inclusion sending $\vq \in \conf{n}{E}$ to the map $\vq\colon \vn \to E$ def\/ined by $\vq(j) = \vq_j$ for each $j\in \vn$.

Now, let $\tilde\vn$ denote the set of all ${n \choose {2}}$ subsets in $\vn$ with two elements: $\tilde\vn=\{\{1,2\},\{1,3\},\ldots$, $\{n-1,n\}\}$. Let $C^1$ denote the vector space of all maps from $\tilde\vn$ to~$E$:
\begin{gather*}
C^1 = \{ \vq \colon \tilde\vn \to E \}.
\end{gather*}
It is isomorphic to $E^{\tilde n}$, where $\tilde n={n \choose {2}}$. Note that if $E^{\vn^2}$ denotes that vector space of all maps $\vq \colon \vn^2 \to E$, where $\vn^2 = \vn\times \vn$ (and hence if $\vq \in E^{\vn^2}$, we can denote $\vq_{ij} = \vq( (i,j) ) \in E$), there is an embedding $C^1 \subset E^{\vn^2}$, by sending an element $\vq\in C^1$ to the map $\vq'\colon \vn^2 \to E$ def\/ined by
\begin{gather*}
\vq'_{ij} =
\begin{cases}
\vq(\{i,j\}) & \text{if $i<j$}, \\
-\vq(\{i,j\}) & \text{if $i>j$}, \\
0 & \text{if $i=j$.}
\end{cases}
\end{gather*}
In fact, we are identifying elements in $C^1$ with the skew-symmetric elements in $E^{\vn^2}$ (that is, maps $\vq_{ij} + \vq_{ji} = \zero \in E$). Given $\vq \in C^1$ with an abuse of notation we will write $\vq_{ij}$ instead of~$\vq'_{ij}$, and~$ij$ instead $(i,j)$.

If $K$ is an abstract simplicial complex, recall that the simplicial chain complex of~$K$ with real coef\/f\/icients, denoted by $C_*(K;\RR)$, is def\/ined as follows: for each $k\in \ZZ$, the chain group $C_k(K;\RR)$ is the vector space of all the $\RR$-linear combinations of $k$-dimensional simplexes of~$K$; the boundary homomorphism $\partial_k \colon C_k(K;\RR) \to C_{k-1}(K;\RR)$ is def\/ined as $\partial_k(\sigma ) = \sum\limits_{j=0}^k (-1)^j \sigma d_j $ for each $k$-simplex $\sigma$ of~$K$ and $0$ otherwise, where $d_j$ is the $j$-th face map. More precisely, all simplices in~$K$ can be ordered, and elements in $C_k(K;\RR)$ will be linear combinations of ordered $k$-simplices in~$K$. An ordered $k$-simplex with vertices $x_0,\ldots, x_{k}$ will be denoted either as $[x_0,\ldots, x_k]$ or simply as $x_0\ldots x_k$. With this notation, the $j$-th face map sends $\sigma = [x_0,\ldots,x_j,\ldots, x_k]$ to $\sigma d_j = [ x_0,\ldots, \widehat x_j, \ldots, x_k]$, where $\widehat x_j$ means that the $j$-th element is canceled.

By taking homomorphisms valued in an $\RR$-vector space $E$, the chain complex $C_k(K;\RR)$ yields the simplicial cochain complex with coef\/f\/icients in $E$: the cochain groups are def\/ined as all the linear homomorsphisms $C^k(K;E) = \hom_\RR(C_k(K;\RR), E)$, and the coboundary homomorphisms $\delta^k \colon C^k(K;E) \to C^{k+1}(K;E)$ are def\/ined for each $k$ by
\begin{gather*}
\delta^k( \eta ) = \eta \partial_{k+1} \colon \ C_{k+1}(K;\RR) \to E
\end{gather*}
for each cochain $\eta \colon C_{k}(K;\RR) \to E$. The kernel of $\delta^k$ is the group of cocycles, and it is denoted as $Z^k(K;E) = \ker \delta^k \subset C^k(K;E)$.

Now, let $\Delta^{n-1}$ denote the standard (abstract) simplex with~$n$ vertices $\{1,2,\ldots, n\}$. Then the vector spaces $C^k$ def\/ined above for $k=0,1$ are exactly the groups of $k$-dimensional simplicial cochains (with coef\/f\/icients in the vector space~$E$) of the simplicial complex $\Delta^{n-1}$: $C^0 = C^0(\Delta^{n-1};E)$ and $C^1 = C^1(\Delta^{n-1};E)$. A $0$-simplex of $\Delta^{n-1}$ is simply an element $j \in \vn $, and hence a~$0$-dimensional cochain is an $n$-tuple $\vq_j$, i.e., a~map $\vq \colon \vn \to E$. Furthermore, a~$1$-di\-men\-sional cochain is a~map $\vq$ def\/ined with values in $E$ and as domain the set of $1$-dimensional simplices of $\Delta^{n-1}$, i.e., pairs $ij$ with $1\leq i<j\leq n$.

In simpler terms, for each $i,j$ such that $1\leq i < j \leq n$, let $\vq_{ij} \in E$ denote the $ij$-the component of a vector in $E^{\tilde n}$, and for $i> j$, the variable $\vq_{ij}$ is def\/ined by the property that $\forall\, i,j$, $\vq_{ij} + \vq_{ji} = \zero$.

The coboundary operator $\delta^0 \colon C^0 \to C^1$ is the map def\/ined by $\delta^0 \vq = \vq \partial_1$ for each $\vq \in C^0$, and hence
\begin{gather*}
\delta^0(\vq) ({i,j}) = \vq_j - \vq_i \in E
\end{gather*}
for all $i$, $j$. In fact, for each $\vq \colon \vn \to E$, $\delta^0(\vq) (ij) = \vq \partial_1 [i,j] = \vq( j-i) = \vq_j - \vq_i$. For $k=1$, the coboundary operator is def\/ined as $\delta^1\colon C^1 \to C^2$ as
\begin{gather*}
\delta^1(\vq)(ijk) = \vq \partial_2(ijk) = \vq( jk - ik + ij) = \vq_{ij} + \vq_{jk} + \vq_{ki}.
\end{gather*}

Moreover, since the simplex $\Delta^{n-1}$ is contractible, its cohomology groups are trivial except for $k-0$, and therefore for each $k \geq 0$
\begin{gather*}
Z^{k+1}\big(\Delta^{n-1};E\big) = \ker \delta^{k+1} = \im \delta^{k}.
\end{gather*}
With an abuse of notation, when not necessary the subscript of $\partial_k$ and the supscript in $\delta^k$ will be omitted.

For each $\vq \in C^1$, let $\vQ$ be def\/ined as $\vQ_{jk} = \frac{\vq_{jk}}{\eucnorm{\vq_{jk}}^{\alpha+2}}$. Note that $\vQ_{jk} = \Psi_\gamma (\vq_{jk})$
with $\gamma=\alpha+2$ where $\Psi_\gamma(\vx) = \frac{\vx}{\eucnorm{\vx}^\gamma}$ for each $\vx\in E$. It turns out that the map $\Psi_\gamma\colon E\setminus\{\zero\} \to E \setminus \{\zero\}$ is a~dif\/feomorphism with inverse $\Psi_{\hat\gamma}$ where $\hat\gamma = \frac{\gamma}{\gamma-1}$.

Consider now that if one def\/ines $\vq_{ij}=\vq_i - \vq_j$, one can read equation \eqref{eq:cc1} as (as a consequence of equation~\eqref{eq:masses1})
\begin{gather*}
\lambda \vq_j=-\alpha \sum_{k\neq j} m_k\frac{\vq_{jk}}{\eucnorm{\vq_{jk}}^{\alpha+2}}=-\alpha\sum_{k\neq j} m_k \vQ_{jk},
\end{gather*}
and hence
\begin{align}
-\frac{\lambda}{\alpha} \vq_{ij} & =\sum_{k\neq i} m_k \vQ_{ik}-\sum_{k\neq j} m_k \vQ_{jk} =
\sum_{k\not\in\{i,j\}} m_k (\vQ_{ik} + \vQ_{kj})+ m_j \vQ_{ij} + m_i \vQ_{ij} \nonumber\\
& =\sum_{k\not\in\{i,j\}} m_k (\vQ_{ik} + \vQ_{kj} + \vQ_{ji})-\sum_{k\not\in\{i,j\}} m_k \vQ_{ji}+ m_j \vQ_{ij} + m_i \vQ_{ij} \nonumber\\
&=\sum_{k\not\in\{i,j\}} m_k (\vQ_{ik} + \vQ_{kj} + \vQ_{ji})+\left(\sum_k m_k\right) \vQ_{ij} \nonumber\\
\label{eq:qij}
\iff -\frac{\lambda}{\alpha} \vq_{ij}& = \sum_{k\not\in\{i,j\}} m_k (\vQ_{ik} + \vQ_{kj} + \vQ_{ji})+\vQ_{ij}.
\end{align}

\begin{propo}\label{propo:projection}
The $($linear$)$ map $P_m \colon C^1 \to C^1$ defined by
\begin{gather*}
\left( P_m(\vQ) \right)_{ij} = \sum_{k\not\in\{i,j\}} m_k (\vQ_{ik} + \vQ_{kj} + \vQ_{ji})+\vQ_{ij}
\end{gather*}
is a projection from $C^1$ onto $Z^1\subset C^1$, where $Z^1=\ker \delta^1 \colon C^1 \to C^2$ is the subspace of $1$-cocycles.
\end{propo}
\begin{proof}
Consider the homomorphism $\pi_m$ def\/ined on the vector space of simplicial $1$-chains $C_1(\Delta^{n-1},\RR)$ with real coef\/f\/icient, as
\begin{gather*}
\pi_m([i,j]) = [i,j] - \sum_{k\neq \{i,j\}} m_k \partial_2([i,j,k]),
\end{gather*}
where $\partial_2 \colon C_2 \to C_1$ is the boundary homomorphism in dimension~$2$. Then for any $\vQ \in C^1$ and any $i$, $j$ with $i\neq j$
\begin{gather*}
P_m(\vQ)[i,j] = \vQ ( \pi_m [i,j]) .
\end{gather*}
For each $2$-simplex $[a,b,c]$ of $\Delta^{n-1}$ one has $\partial_2 ([a,b,k] + [b,c,k] + [c,a,k]) = \partial_2 [a,b,c]$ for each $k\neq \{a,b,c\}$, and hence
\begin{gather*}
\pi_m \partial_2 [a,b,c] = \pi_m([a,b] + [b,c] + [c,a]) =
[a,b] - \sum_{k\neq \{a,b\}} m_k \partial_2 [a,b,k]\\
\hphantom{\pi_m \partial_2 [a,b,c] =}{} +
[b,c] - \sum_{k\neq \{b,c\}} m_k \partial_2 [b,c,k] + [c,a] - \sum_{k\neq \{c,a\}} m_k \partial_2 [c,a,k]\\
\hphantom{\pi_m \partial_2 [a,b,c]}{} =
\partial_2 [a,b,c]- \sum_{k\not\in\{a,b,c\}} m_k \partial_2 [a,b,c]-
m_c \partial_2 [a,b,c] -m_a \partial_2 [b,c,a] -m_b \partial_2 [c,a,b]\\
\hphantom{\pi_m \partial_2 [a,b,c]}{} = \partial_2 [a,b,c] - \left( \sum_k m_k \right) \partial_2 [a,b,c] = 0\\
\implies \pi_m \partial_2 = 0.
\end{gather*}

As a consequence, $\pi_m$ is a projection, since for each $i$, $j$, $i\neq j$,
\begin{gather*}
\pi_m^2[i,j] =\pi_m\left([i,j] - \sum_{k\not\in \{i,j\}} m_k \partial_2([i,j,k])\right) = \pi_m[i,j] - \sum_{k\not\in\{i,j\}} m_k \pi_m\partial_2[i,j,k] =
\pi_m[i,j].
\end{gather*}
Therefore, also $P_m\colon C^1 \to C^1$ is a projection
\begin{gather*}
P^2_m ( \vQ ) [i,j] = P_m ( \vQ) (\pi_m [i,j] ) = \vQ ( \pi_m^2 [i,j] ) = \vQ ( \pi_m [i,j] ) = P_m(\vQ)[i,j].
\end{gather*}

Moreover, since
\begin{gather}\label{eq:partialpim}
\partial_1 \pi_m = \partial_1
\end{gather}
it follows that the projection $P_m$ is onto the subspace of all $1$-cocycles in $C^1$, denoted in short by~$Z^1$. In fact, since $\pi_m \partial_2 =0$,
\begin{gather*}
\delta^1 P_m \vQ = P_m \vQ \partial_2 = \vQ \pi_m \partial_2 = 0 \implies \Ima(P_m) \subset Z^1
\end{gather*}
and, by \eqref{eq:partialpim}, for each cocycle $\vz \in Z^1 \iff \vz = \delta^0 \vx $ one has
\begin{gather*}
P_m \vz = P_m \delta^0 \vx = P_m \vx \partial_1 = \vx \partial_1 \pi_m = \vx \partial_1 = \delta^0 \vx = \vz
\end{gather*}
and hence $\Ima(P_m) \supset Z^1$. We can conclude, as claimed, that $\Ima(P_m) = Z^1$.
\end{proof}

As examples, for $d=1$ and $n=3,4$ the matrices of the projection $P_m$ are
\begin{gather*}
\begin{bmatrix}
m_{1}+m_{2}& m_{3} & -m_{3} \\
 m_{2} &m_{1}+m_{3}& m_{2} \\
 -m_{1} & m_{1} &m_{2}+m_{3}
\end{bmatrix},
\\
\begin{bmatrix}
m_{1}+m_{2}& m_{3} & m_{4} & -m_{3} & -m_{4} &0 \\
 m_{2} &m_{1}+m_{3}& m_{4} & m_{2} &0& -m_{4} \\
 m_{2} & m_{3} &m_{1}+m_{4}&0& m_{2} & m_{3} \\
 -m_{1} & m_{1} &0&m_{2}+m_{3}& m_{4} & -m_{4} \\
 -m_{1} &0& m_{1} & m_{3} &m_{2}+m_{4}& m_{3} \\
0& -m_{1} & m_{1} & -m_{2} & m_{2} &m_{3}+m_{4}
\end{bmatrix}.
\end{gather*}
In fact, for $n=3$ the space of cochains $C^1$ has standard coordinates $\vQ_{ij}$ for $ij\in \{12,13,23\}$, and by Proposition~\ref{propo:projection} the projection $P_m$ in these coordinates is def\/ined by
\begin{gather*}
\left( P_m(\vQ) \right)_{12} =\sum_{k\not\in\{1,2\}} m_k (\vQ_{1k} + \vQ_{k2} + \vQ_{21})+\vQ_{12} =m_3 ( \vQ_{13} + \vQ_{32} + \vQ_{21}) +\vQ_{12},\\
\left( P_m(\vQ) \right)_{13} =\sum_{k\not\in\{1,3\}} m_k (\vQ_{1k} + \vQ_{k3} + \vQ_{31})+\vQ_{13} =m_2(\vQ_{12} + \vQ_{23} + \vQ_{31}) + \vQ_{13},\\
\left( P_m(\vQ) \right)_{23} =\sum_{k\not\in\{2,3\}} m_k (\vQ_{2k} + \vQ_{k3} + \vQ_{32})+\vQ_{23} =m_1( \vQ_{21}+\vQ_{13} + \vQ_{32} ) + \vQ_{23},
\end{gather*}
from which it follows that
\begin{gather*}
\left( P_m(\vQ) \right)_{12} =(1-m_3) \vQ_{12}+ m_3 \vQ_{13}- m_3 \vQ_{23},\\
\left( P_m(\vQ) \right)_{13} =m_2\vQ_{12}+ (1-m_2)\vQ_{13}+ m_2 \vQ_{23},\\
\left( P_m(\vQ) \right)_{23} =- m_1\vQ_{12}+ m_1 \vQ_{13}+ (1-m_1)\vQ_{23}.
\end{gather*}
The same argument yields the matrix for $n=4$, in coordinates $\vQ_{ij}$ for $ij$ in the order $12$, $13$, $14$, $23$, $24$, $34$.

Consider the following scalar product on $C^1$, similar to the mass-metric on the conf\/iguration space: for $\vv,\vw \in C^1$,
\begin{gather}\label{mmscalar}
\mscalar{\vv}{\vw} =\sum_{i<j} m_i m_j \left(\vv_{ij} \cdot \vw_{ij}\right),
\end{gather}
where as above the dot denotes the standard $d$-dimensional scalar product in $E$. It is the mass-metric on $C^1$, and as above $\mnorm{\vv}^2 = \mscalar{\vv}{\vv}$. It follows that
\begin{gather*}
\mscalar{ \vv}{ P_m (\vw) } =\sum_{i<j}m_i m_j \left( \vv_{ij} \cdot \left( P_m(\vw) \right)_{ij} \right)\\
\hphantom{\mscalar{ \vv}{ P_m (\vw) }}{} =\sum_{i<j} m_i m_j \left( \vv_{ij} \cdot \left(\sum_{k\not\in\{i,j\}} m_k (\vw_{ik} + \vw_{kj} + \vw_{ji})+\vw_{ij} \right) \right) \\
\hphantom{\mscalar{ \vv}{ P_m (\vw) }}{} = \sum_{i<k} \sum_{k\not\in\{i,j\}}m_i m_j m_k ( \vv_{ij} \cdot (\vw_{ik}+\vw_{kj} + \vw_{ji}) )+
\sum_{i<j} m_i m_j \vv_{ij} \cdot \vw_{ij} \\
\hphantom{\mscalar{ \vv}{ P_m (\vw) }}{} = \sum_{a<b<c} m_a m_b m_c\big(\vv_{ab}\cdot ( \vw_{ac} + \vw_{cb} + \vw_{ba}) \\
\hphantom{\mscalar{ \vv}{ P_m (\vw) }=}{}+ \vv_{ac} \cdot ( \vw_{ab} + \vw_{bc} + \vw_{ca})+\vv_{bc} \cdot ( \vw_{ba} + \vw_{ac} + \vw_{cb})\big)+\sum_{i<j} m_i m_j \vv_{ij} \cdot \vw_{ij} \\
\hphantom{\mscalar{ \vv}{ P_m (\vw) }}{}= - \sum_{a<b<c} m_a m_b m_c\left((\vv_{ab}+\vv_{bc}+\vv_{ca})\cdot ( \vw_{ab} + \vw_{bc} + \vw_{ca})
\right)\\
\hphantom{\mscalar{ \vv}{ P_m (\vw) }=}{}+ \sum_{i<j} m_i m_j \vv_{ij} \cdot \vw_{ij} \\
\hphantom{\mscalar{ \vv}{ P_m (\vw) }}{}=\sum_{i<k} \sum_{k\not\in\{i,j\}}m_i m_j m_k ( (\vv_{ik}+\vv_{kj} + \vv_{ji}) \cdot \vw_{ij} )+
\sum_{i<j} m_i m_j \vv_{ij} \cdot \vw_{ij} \\
\hphantom{\mscalar{ \vv}{ P_m (\vw) }}{} =\sum_{i<j} m_i m_j \left( \left(\sum_{k\not\in\{i,j\}} m_k (\vv_{ik} + \vv_{kj} + \vv_{ji})+
\vv_{ij} \right) \cdot \vw_{ij} \right)\\
\hphantom{\mscalar{ \vv}{ P_m (\vw) }}{}= \mscalar{ P_m (\vv)}{ \vw },
\end{gather*}
hence the following proposition holds.

\begin{propo}
The projection $P_m\colon C^1 \to Z^1\subset C^1$ is orthogonal $($self-adjoint$)$ with respect to the scalar product $\mscalar{-}{-}$ in~\eqref{mmscalar} defined on~$C^1$.
\end{propo}

Now, consider the subspace $X\subset \conf{n}{E}$ of all conf\/igurations with center of mass in $\zero$: $X=\{ \vq \in \conf{n}{E} \colon \sum_j m_j \vq_j = \zero \}$, i.e., of all $\vq \in C^0$ such that $\vq \sum_j m_j [j] = \zero$. The coboundary morphism $\delta^0_{|X}\colon X \subset C^0 \to C^1$ induces an isomorphism $\delta^0_{|X} \colon X \to Z^1$. Moreover, since if $\vq \in X$ then
\begin{gather*}
2 \sum_{i<j} m_im_j \eucnorm{\vq_i -\vq_j}^2 =\sum_{i,j} m_i m_j \eucnorm{\vq_i -\vq_j}^2 = \sum_{i,j} m_i m_j \left(
\eucnorm{\vq_i}^2 - 2\vq_i \cdot \vq_j + \eucnorm{\vq_j}^2\right) \\
\hphantom{2 \sum_{i<j} m_im_j \eucnorm{\vq_i -\vq_j}^2}{} =
2 \sum_{i} m_i \eucnorm{\vq_i}^2 - 2 \left( \sum_{j} m_j \vq_j \right)^2 =2 \sum_{i} m_i \eucnorm{\vq_i}^2
\end{gather*}
the isomorphism $\delta^0_{|X}\colon X \to Z^1$ is an isometry, where $X$ and $Z^1$ have the mass-metrics. Explicitly, for each $\vq \in X$,
\begin{gather*}
\mnorm{\vq} = \mnorm{ \delta^0 (\vq)},
\end{gather*}
where the two norms with the same symbol, with an abuse of notation, are actually dif\/ferent norms in $C^0$ and $C^1$ respectively.

Furthermore, the potential $U$ is the composition of the restriction to $X$ of the coboundary map $\delta^0$ with the map $\tilde U\colon C^1 \to \RR$ (partially) def\/ined by
\begin{gather*}
\tilde U(\vq) =\sum_{i<j} m_im_j \eucnorm{\vq_{ij}}^{-\alpha},
\end{gather*}
as illustrated in the following diagram
\begin{gather*}
\xymatrix@C+24pt@R+24pt{C^0 \ar[r]^{\delta^0} & C^1 \ar[r]^{\tilde U} \ar@/_1em/[d]_{P_m} & \RR \\
X \ar@{ >->}[u] \ar[r]_{\delta^0_{|X} }^{\cong} & Z^1. \ar@{ >->}@/_1em/[u]
}
\end{gather*}

Now, recall that (condition \textbf{(CC4)}) a conf\/iguration $\vq\in \conf{n}{E} \subset C^0$ is a central conf\/iguration is and only if it is a~critical points of the map $\mnorm{\vq}^2+ U(\vq)$, def\/ined on $\conf{n}{E}$. It is easy to see that this is equivalent to say that $\vq$ is a critical point of the map $\mnorm{\vq}^2+ U(\vq)$ restricted to~$X$. But this means that $\delta^0_{|X}$ sends central conf\/igurations in~$X$ to all the critical points of the map $\mnorm{\tilde \vq}^2+ \tilde U(\tilde \vq)$ (def\/ined on~$C^1$) restricted to the space of $1$-cocycles $Z^1$.

Hence, the following theorem holds.
\begin{theo}\label{theo:Utilde} Central configurations are critical points of the function partially defined as $C^1 \to \RR$
\begin{gather*}
\vq \mapsto \sum_{i<j} m_im_j\left( \eucnorm{\vq_{ij}}^{-\alpha} + \eucnorm{\vq_{ij}}^2 \right)
\end{gather*}
restricted to the space of $1$-cocycles $Z^1\subset C^1$.

A co-chain $\vq \in C^1$ is a central configuration if and only if there exists $\lambda \in \RR$ such that $\lambda \vq = P_m( \Psi(\vQ))$, where $\vQ_{ij} = \frac{\vq_{ij}}{|\vq_{ij}|^{\alpha+2}}$ for each $i$, $j$ and $P_m \colon C^1 \to Z^1\subset C^1$ is the orthogonal projection defined in Proposition~{\rm \ref{propo:projection}}, which sends $C^1$ onto the space of $1$-cocycles.
\end{theo}

\begin{Remark}Since the function $r^{-\alpha} + r^2$ is convex on $(0,\infty)$, Theorem~\ref{theo:Utilde} implies that the restriction of~$\tilde U$ to each component of $Z^1$ minus collisions is convex for $d=1$, and hence one derives the existence (and uniqueness) of Moulton collinear central conf\/igurations.
\end{Remark}

\section{Hessians and indices}

Let $P\in \conf{n}{E}$ be a central conf\/iguration, with mass-norm $r=\mnorm{P}$. As in the case $r=1$, seen in \textbf{(CC1)}, it is a critical point of the restriction of the potential function $U$ to the inertia ellipsoid $S= \{ \vq \in \conf{n}{E} \colon \mnorm{\vq} = r \}$. As such, its Morse index is the Hessian of the restriction $U|_{S}$, which is a bilinear form def\/ined on the tangent space $T_P S$. The Hessian of $f=U|_{S}$ at a critical point $P\in S$ can be computed in general as $D^2f(P)[\vu(P),\vv(P)] = ((D_\vu D_\vv f - D_{D_\vu \vv}) f) (P) = (D_\vu (D_\vv f))(P) $, where $\vu $ and $\vv$ are vector f\/ields on~$S$ (see, e.g., \cite[formula~(1), p.~343]{Lang1999}). This yields the well-known formula of the Hessian in terms of second derivatives with respect to a~local chart (see also \cite[Proposition~2.8.8, p.~136]{MR3469182}). Given the mass-metric, the Hessian can be written as $D^2f[\vu,\vv] = \mscalar{ D_\vu \hatmgrad f }{ \vv }$ or as the self-adjoint endomorphism $T_PS \to T_PS$ def\/ined by $\vu \in T_PS \mapsto D_\vu \hatmgrad f(P) \in T_PS$, where $\hatmgrad$ denotes the gradient induced by the mass-metric restricted to~$S$ (hence $\hatmgrad f (P)$ is the projection of $\mgrad U(P)$ to $T_PS$, orthogonal with respect to $\mscalar{-}{-}$).

If $\vN$ denotes a vector normal to the tangent space $T_PS$ (such as $P-O$, where $O$ denotes the origin of the Euclidean space~$E$), the projection $\hatmgrad f(P)$ is equal to $\mgrad U - \frac{\mscalar{\mgrad U}{\vN}}{\mnorm{\vN}^2} \vN $ evaluated at $P$. If $\vN=P-O$, by Euler formula $\mscalar{\mgrad U}{\vN} = -\alpha U$, and because $P$ is a~critical point of $f=U|_S$ and $\vu \in T_PS$ the derivative $D_\vu \frac{\alpha U(\vq)}{\mnorm{\vq}^2}$ vanishes at~$P$, and hence
\begin{gather*}
D_\vu \hatmgrad f(P) =D_\vu \left(\mgrad U(\vq) + \frac{\alpha U(\vq)}{\mnorm{\vq}^2} \vN\right) (P) =
D_\vu (\mgrad U)(P)+ \frac{\alpha U(P)}{\mnorm{P}^2}D_\vu(\vq) \\
\hphantom{D_\vu \hatmgrad f(P)}{} = D_\vu\left(\mgrad U-\lambda\mgrad \frac{\mnorm{\vq}^2}{2}\right),
\end{gather*}
where $\lambda$ is as above the constant $-\frac{\alpha U(P)}{\mnorm{P}^2}$. Hence the following lemma holds.

\begin{Lemma}\label{propo:hessian1}
If $P$ is a critical point of the restriction $U|_{S_r}$, with the inertia ellipsoid $S_r =\{ \vq \in \conf{n}{E} \colon \mnorm{\vq} = r \}$ and with $\lambda$ defined as $\lambda = -\frac{\alpha U(P)}{\mnorm{P}^2} = -\frac{\alpha U(P)}{r^2}$, then~$P$ is a critical point of the function $U(\vq)-\frac{\lambda}{2} \mnorm{\vq}^2$; moreover, the Hessian of $U|_{S_r}$ at $P$ is the restriction to the tangent space $T_P{S_r}$ of the Hessian of the map $U(\vq)-\frac{\lambda}{2} \mnorm{\vq}^2$ defined on $\conf{n}{E}$, evaluated $P$.
\end{Lemma}

\begin{propo}\label{propo:dexter}
If $P\in \conf{n}{E}$ is a central configuration, then the Morse index at $P$ of the restriction $f=U|_{S_r}$ is equal to the Morse index at~$P$ of the function $F(\vq) = U(\vq)-\frac{\lambda}{2} \mnorm{\vq}^2$, where $\lambda$ and $S_r$ are as above. Furthermore, the direction parallel to $P-O$ is an eigenvector of the Hessian of~$F$, with (positive) eigenvalue equal to $-\lambda(\alpha+2)>0$.
\end{propo}
\begin{proof}Since $\mgrad{U(\vq)}$ is homogeneous of degree $-\alpha-1$, if $\vN = P-O$ one has $D_\vN( \mgrad{U} )(P) = -(\alpha+1) \mgrad{U(P)} = -\lambda (\alpha+1) \vN$. Therefore
\begin{gather*}
D_\vN\left(\mgrad U-\lambda\mgrad \frac{\mnorm{\vq}^2}{2}\right) (P)=-\lambda (\alpha+1) \vN - \lambda \vN=-\lambda (\alpha+2) \vN.\tag*{\qed}
\end{gather*}
\renewcommand{\qed}{}
\end{proof}

Now, consider the function $f\colon C^1 \to \RR$ def\/ined on cochains in Theorem~\ref{theo:Utilde} as
\begin{gather*}
f(\vq) = \sum_{i<j} m_im_j\left( \eucnorm{\vq_{ij}}^{-\alpha} + \eucnorm{\vq_{ij}}^2 \right).
\end{gather*}
The following proposition links its Hessian with the Hessian of the function~$F$ of Proposition~\ref{propo:dexter}, for $\lambda = -2$.

\begin{propo}\label{propo:hessian2}
Let $\vq\in\conf{n}{E}$ be a central configuration which is a critical point of the function $F(\vq) = U(\vq) + \mnorm{\vq}^2$ in $C^0$ $($and hence $\vq\in X)$. Let $H$ be the Hessian of $F$ at~$\vq$ $($with respect to the mass-metric in~$C^0)$, and~$ \tilde H$ the Hessian matrix of the composition $f \circ P_m $ at $\delta^0(\vq)\in Z^1 \subset C^1$ $($with respect to the mass-metric in~$C^1)$. Then the non-zero eigenvalues of~$\tilde H$ are the same as the non-zero eigenvalues of~$H$, except for the eigenvalue~$2$ occurring in~$H$ with multiplicity $\dim E$ $($which corresponds to the group of translations in~$E$, or equivalently the orthogonal complement of~$X$ in~$C^0)$.
\end{propo}
\begin{proof} Since $U$ is invariant with respect to translations in $C^0$, $H$ has the autospace $\vq_1=\vq_2=\dots=\vq_n \subset C^0$ (which is the tangent space of the group of translations acting on $\conf{n}{E}$, and is orthogonal to $X$ with respect to the mass-metric), over which $D^2U$ vanishes and $D^2 \mnorm{\vq}^2 = 2$; hence it is an eigenspace with eigenvalue~$2$. The rest of eigenvalues of~$H$ correspond via the isometric embedding $\delta^0_{|X}$ to eigenvalues of the restriction of $f$ to $Z^1$, and hence to the eigenvalues in $Z^1$ of the composition $f \circ P_m$. The orthogonal complement of $Z^1$, which is the kernel of~$P_m$, yields zero eigenvalues to $\tilde H$.
\end{proof}

\section{Simple proofs of some corollaries}

Equations \eqref{eq:qij} can be written as the following:
\begin{gather}\label{eq:qij2}
\frac{\lambda}{\alpha} \vq_{ij} +\vQ_{ij} =\sum_{k\not\in\{i,j\}} m_k (\vQ_{ij} + \vQ_{jk} + \vQ_{ki} ).
\end{gather}

Now, consider for each triple $i$, $j$, $k$ the corresponding term $\vQ_{ijk} = \vQ_{ij} + \vQ_{jk} + \vQ_{ki}$. We give some very simple proofs to some well-known propositions (actually generalizing them to any homogeneity $\alpha$), that follow from the following simple geometric lemma.

\begin{Lemma}\label{lemma:ABC}
Let $\vq_1$, $\vq_2$ and $\vq_3$ be three non-collinear points in $E$. Then $\vQ_{123} = \zero$ if and only if $\vq_1$, $\vq_2$ and $\vq_3$ are vertices of an equilateral triangle.

Furthermore, there exists $c\in \RR$ such that $\vQ_{123} = c \vq_{12}$ if and only if $\eucnorm{\vq_{13}} = \eucnorm{\vq_{23}}$, that is, if and only if the triangle with vertices in $\vq_1$, $\vq_2$ and $\vq_3$ is isosceles in $\vq_3$.
\end{Lemma}
\begin{proof}If $\vq_1$, $\vq_2$ and $\vq_3$ are not collinear (in~$E$), then the dif\/ferences $\vq_{12}$, $\vq_{13}$ and $\vq_{23}$ generate a plane. Since $\vq_{12} + \vq_{23} + \vq_{31} = \zero$, $\vQ_{123}=\zero$ implies
\begin{gather*}
\vQ_{123} = \vQ_{12} + \vQ_{23} + \vQ_{31} =\frac{\vq_{12}}{\eucnorm{\vq_{12}}^{\alpha+2}} +\frac{\vq_{23}}{\eucnorm{\vq_{23}}^{\alpha+2}} +
\frac{\vq_{31}}{\eucnorm{\vq_{31}}^{\alpha+2}}
 = \zero = \vq_{12} + \vq_{23} +\vq_{31} .
\end{gather*}
By taking barycentric coordinates in the plane generated by the three points, it follows that $\vQ_{123} = \zero$ if and only if $\eucnorm{\vq_{12}}^{\alpha+2} = \eucnorm{\vq_{23}}^{\alpha+2} = \eucnorm{\vq_{31}}^{\alpha+2}$, that is, if and only if the three points are vertices of an equilateral triangle.

If $c_1$, $c_2$ and $c_3$ are three non-zero real numbers such that $c_3 \vq_{12} + c_1 \vq_{23} + c_2 \vq_{31} = c \vq_{12}$, then $(c_3-c) \vq_{12} + c_1 \vq_{23} + c_2 \vq_{31} = \zero $, and as above this implies $c_3-c=c_1=c_2$. Hence if $\vQ_{123} = c \vq_{12}$, it holds that $\eucnorm{\vq_{23}}^{\alpha+2} = \eucnorm{\vq_{13}}^{\alpha+2} $ as claimed.
\end{proof}

\begin{coro}
For $n= 3$, $d\geq 2$ and $\alpha>0$, the only non-collinear central configuration is the equilateral configuration.
\end{coro}
\begin{proof}
Equation \eqref{eq:qij2} implies that, if $\vQ_{ijk} \neq \zero$, then the conf\/iguration is collinear (since \eqref{eq:qij2} implies there exist three real numbers $c_{12}$, $c_{23}$, $c_{31}$ such that $c_{12}\vq_{12} = m_3 \vQ_{123}$, $c_{23}\vq_{23} = m_1 \vQ_{231}$, $c_{31}\vq_{31} = m_2 \vQ_{312}$, and it is easy to see that $\vQ_{123} = \vQ_{231} = \vQ_{312}$). Therefore, if the conf\/iguration is not collinear, $\vQ_{ijk} = \zero$, and by Lemma~\ref{lemma:ABC} the conf\/iguration is an equilateral triangle.
\end{proof}

Another easy consequence of Lemma~\ref{lemma:ABC} is the following proposition (see \cite{MR3428762,mb1932, MR2052612} for its importance in estimating the number non-degenerate planar central conf\/igurations of four bodies).

\begin{coro}\label{coro:bartky}
For $n=4$, $d\geq 2$ and $\alpha>0$, if a central configuration has three collinear bodies, then it is a collinear configuration.
\end{coro}
\begin{proof}
Assume that $\vq_1$, $\vq_2$ and $\vq_3$ are collinear, and $\vq_4$ is not. Then, equation \eqref{eq:qij2} implies that for suitable real numbers $c_{12}$, $c_{23}$ and $c_{31}$ the following equations hold:
\begin{gather*}
c_{12} \vq_{12} = m_3 \vQ_{123} + m_4 \vQ_{124}, \qquad c_{23} \vq_{23} = m_1 \vQ_{231} + m_4 \vQ_{234},\\
c_{31} \vq_{31} = m_2 \vQ_{312} + m_4 \vQ_{314}.
\end{gather*}
This implies that there are $\tilde c_{12}$, $\tilde c_{23}$ and $\tilde c_{31}$
such that
\begin{gather*}
\vQ_{124} = \tilde c_{12} \vq_{12}, \qquad \vQ_{234} = \tilde c_{23} \vq_{23}, \qquad \vQ_{314} = \tilde c_{31} \vq_{31},
\end{gather*}
and by Lemma~\ref{lemma:ABC} this implies that $\eucnorm{\vq_{14}}=\eucnorm{\vq_{24}} = \eucnorm{\vq_{34}}$, which is not possible since $\vq_1$, $\vq_2$ and~$\vq_3$ are collinear.
\end{proof}

Corollary \ref{coro:bartky} can be easily generalized to arbitrary $n$ as follows:
\begin{coro}\label{coro:bartky2}
For $n\geq 4$, $d\geq 2$ and $\alpha>0$, if $n-1$ of the bodies in the central configuration are collinear, then all of them are.
\end{coro}

\begin{coro} For $n\geq 4$, $d\geq 3$ and $\alpha>0$, if the first $n-1$ bodies $\vq_1, \ldots, \vq_{n-1}$ in a central configuration belong to a plane $\pi\subset E$, and the $n$-th body $\vq_n$ does not belong to the plane $\pi$, then the distance between $\vq_n$ and any $\vq_j$ does not depend on $j=1,\ldots, n-1$, i.e., there exists $c>0$ such that $\eucnorm{\vq_n - \vq_j} = c$ for all $j<n$. Hence the $n-1$ coplanar bodies are cocircular.
\end{coro}
\begin{proof}
For each $i,j\leq n-1$ there exists $c_{ij} \in \RR$ such that
\begin{gather*}
c_{ij} \vq_{ij} = \sum_{k\not\in\{i,j\}} m_k \vQ_{ijk} = \sum_{k\not\in\{i,j,n\}} m_k \vQ_{ijk}+ m_n \vQ_{ijn}.
\end{gather*}
The term $\sum\limits_{k\not\in\{i,j,n\}} m_k \vQ_{ijk}$ is parallel to the plane $\pi$, while the sum $c_{ij} \vq_{ij} - m_n \vQ_{ijn}$ is a vector parallel to the plane containing $\vq_i$, $\vq_j$ and $\vq_n$. Being equal, they both need to be parallel to both planes, and hence they are multiples of~$\vq_{ij}$. Therefore, by Lemma~\ref{lemma:ABC}, there exists $\tilde c_{ij}\in \RR$ such that $\vQ_{ijn} = \tilde c_{ij} \vq_{ij}$, and as above this implies that $\eucnorm{\vq_{in}} = \eucnorm{\vq_{jn}}$.
\end{proof}

Pyramidal conf\/igurations for $d=3$ and $\alpha=1$ were studied in f\/irst~\cite{MR1301024}; see also~\cite{MR2108877} for applications to perverse solutions and for the value of the constant~$c$.

\subsection*{Acknowledgements}
Work partially supported by the project ERC Advanced Grant 2013 n.~339958 ``Complex Patterns for Strongly Interacting Dynamical Systems COMPAT''.

\pdfbookmark[1]{References}{ref}
\LastPageEnding


\begin{thebibliography}{99}
\footnotesize\itemsep=0pt

\bibitem{MR3428762}
Albouy A., Open problem 1: are {P}almore's ``ignored estimates'' on the number
 of planar central conf\/igurations correct?, \href{https://doi.org/10.1007/s12346-015-0170-z}{\textit{Qual. Theory Dyn. Syst.}}
 \textbf{14} (2015), 403--406, \href{http://arxiv.org/abs/1501.00694}{arXiv:1501.00694}.

\bibitem{MR1489897}
Albouy A., Chenciner A., Le probl\`eme des {$n$} corps et les distances
 mutuelles, \href{https://doi.org/10.1007/s002220050200}{\textit{Invent. Math.}} \textbf{131} (1998), 151--184.

\bibitem{MR2925390}
Albouy A., Kaloshin V., Finiteness of central conf\/igurations of f\/ive bodies in
 the plane, \href{https://doi.org/10.4007/annals.2012.176.1.10}{\textit{Ann. of Math.}} \textbf{176} (2012), 535--588.

\bibitem{FH}
Fadell E.R., Husseini S.Y., Geometry and topology of conf\/iguration spaces,
 \href{https://doi.org/10.1007/978-3-642-56446-8}{\textit{Springer Monographs in Mathe\-matics}}, Springer-Verlag, Berlin, 2001.

\bibitem{MR1301024}
Fay\c{c}al N., On the classif\/ication of pyramidal central conf\/igurations,
 \href{https://doi.org/10.1090/S0002-9939-96-03135-8}{\textit{Proc. Amer. Math. Soc.}} \textbf{124} (1996), 249--258.

\bibitem{MR2372989}
Ferrario D.L., Planar central conf\/igurations as f\/ixed points, \href{https://doi.org/10.1007/s11784-007-0032-7}{\textit{J.~Fixed
 Point Theory Appl.}} \textbf{2} (2007), 277--291.

\bibitem{Fer2015}
Ferrario D.L., Fixed point indices of central conf\/igurations, \href{https://doi.org/10.1007/s11784-015-0246-z}{\textit{J.~Fixed
 Point Theory Appl.}} \textbf{17} (2015), 239--251, \href{http://arxiv.org/abs/1412.5817}{arXiv:1412.5817}.

\bibitem{MR2207019}
Hampton M., Moeckel R., Finiteness of relative equilibria of the four-body
 problem, \href{https://doi.org/10.1007/s00222-005-0461-0}{\textit{Invent. Math.}} \textbf{163} (2006), 289--312.

\bibitem{ItuMad2015}
Iturriaga R., Maderna E., Generic uniqueness of the minimal {M}oulton central
 conf\/iguration, \href{https://doi.org/10.1007/s10569-015-9642-3}{\textit{Celestial Mech. Dynam. Astronom.}} \textbf{123} (2015),
 351--361, \href{http://arxiv.org/abs/1406.6887}{arXiv:1406.6887}.

\bibitem{Lang1999}
Lang S., Fundamentals of dif\/ferential geometry, \href{https://doi.org/10.1007/978-1-4612-0541-8}{\textit{Graduate Texts in
 Mathematics}}, Vol.~191, Springer-Verlag, New York, 1999.

\bibitem{mb1932}
MacMillan W.D., Bartky W., Permanent conf\/igurations in the problem of four
 bodies, \href{https://doi.org/10.2307/1989432}{\textit{Trans. Amer. Math. Soc.}} \textbf{34} (1932), 838--875.

\bibitem{MR805839}
Moeckel R., Relative equilibria of the four-body problem, \href{https://doi.org/10.1017/S0143385700003047}{\textit{Ergodic
 Theory Dynam. Systems}} \textbf{5} (1985), 417--435.

\bibitem{MR1082871}
Moeckel R., On central conf\/igurations, \href{https://doi.org/10.1007/BF02571259}{\textit{Math.~Z.}} \textbf{205} (1990),
 499--517.

\bibitem{MR3469182}
Moeckel R., Central conf\/igurations, in Central Conf\/igurations, Periodic Orbits,
 and {H}amiltonian Systems, \href{https://doi.org/10.1007/978-3-0348-0933-7_2}{\textit{Adv. Courses Math. CRM Barcelona}},
 Birkh\"auser/Springer, Basel, 2015, 105--167.

\bibitem{MR3069058}
Moeckel R., Montgomery R., Symmetric regularization, reduction and blow-up of
 the planar three-body problem, \href{https://doi.org/10.2140/pjm.2013.262.129}{\textit{Pacific~J. Math.}} \textbf{262} (2013),
 129--189, \href{http://arxiv.org/abs/1202.0972}{arXiv:1202.0972}.

\bibitem{MR2108877}
Ouyang T., Xie Z., Zhang S., Pyramidal central conf\/igurations and perverse
 solutions, \textit{Electron.~J. Dif\-fe\-ren\-tial Equations} (2004), 106,
 9~pages.

\bibitem{MR2052612}
Xia Z., Convex central conf\/igurations for the {$n$}-body problem,
 \href{https://doi.org/10.1016/j.jde.2003.10.001}{\textit{J.~Differential Equations}} \textbf{200} (2004), 185--190.

\end{thebibliography}
\end{document}